\documentclass[12pt]{amsart}

\usepackage{latexsym, amssymb, amsmath,ulem,soul,esint}

\usepackage{amsfonts, graphicx}
\usepackage{graphicx,color}
\newcommand{\be}{\begin{equation}}
\newcommand{\ee}{\end{equation}}
\newcommand{\beq}{\begin{eqnarray}}
\newcommand{\eeq}{\end{eqnarray}}

\newtheorem{thm}{Theorem}[section]

\newtheorem{lma}{Lemma}[section]
\newtheorem{prop}{Proposition}[section]
\newtheorem{cor}{Corollary}[section]

\theoremstyle{remark}

\numberwithin{equation}{section}
\def\tr{\operatorname{tr}}

\def\be{\begin{equation}}
\def\ee{\end{equation}}
\def\bee{\begin{equation*}}
\def\eee{\end{equation*}}
\def\ol{\overline}
\def\lf{\left}
\def\ri{\right}

\def\K{K\"ahler }

\def\KE{K\"ahler-Einstein }
\def\KR{K\"ahler-Ricci }
\def\Ric{\text{\rm Ric}}

\def\cS{\mathcal{S}}

\def\wt{\widetilde}
\def\la{\langle}
\def\ra{\rangle}
\def\p{\partial}

\def\ol{\overline}
\def\heat{\lf(\frac{\p}{\p t}-\Delta_{g(t)}\ri)}
\def\tr{\operatorname{tr}}

\def\e{\varepsilon}
\def\a{{\alpha}}
\def\b{{\beta}}

\def\R{\mathbb{R}}

\def\ddb{\sqrt{-1}\partial\bar\partial}

\begin{document}

\title[]
{K\"ahler manifolds and mixed curvature}

\author{Jianchun Chu}
\address[Jianchun Chu]{School of Mathematical Sciences, Peking University, Yiheyuan Road 5, Beijing, P.R.China, 100871
\newline
\indent
Department of Mathematics, Northwestern University, 2033 Sheridan Road, Evanston, IL 60208}
\email{jianchunchu@math.pku.edu.cn}

\author{Man-Chun Lee$^1$}
\address[Man-Chun Lee]{Department of Mathematics, The Chinese University of Hong Kong, Shatin, N.T., Hong Kong
\newline
\indent
Department of Mathematics, Northwestern University, 2033 Sheridan Road, Evanston, IL 60208
\newline
\indent
Mathematics Institute, Zeeman Building,
University of Warwick, Coventry CV4 7AL
}
\email{mclee@math.cuhk.edu.hk}

\author{Luen-Fai Tam$^2$}
\address[Luen-Fai Tam]{The Institute of Mathematical Sciences and Department of Mathematics, The Chinese University of Hong Kong, Shatin, Hong Kong, China.}
 \email{lftam@math.cuhk.edu.hk}
 \thanks{$^1$Research partially supported by NSF grant DMS-1709894 and EPSRC  grant number P/T019824/1}
\thanks{$^2$Research partially supported by Hong Kong RGC General Research Fund \#CUHK 14301517.}

\renewcommand{\subjclassname}{
  \textup{2010} Mathematics Subject Classification}

\date{\today}

\begin{abstract}

In this work we consider compact K\"ahler manifolds with non-positive mixed curvature which is a  ``convex combination" of Ricci curvature and  holomorphic sectional curvature. We show that in this case, the canonical line bundle is nef. Moreover, if the curvature is negative at some point, then the manifold is projective with canonical line bundle being big and nef. If in addition the curvature is negative, then the canonical line bundle is ample. As an application, we answer  a question of Ni concerning manifolds with negative $k$-Ricci curvature and generalize a result of Wu-Yau and Diverio-Trapani to the conformally K\"ahler case.  We also show that the compact K\"ahler manifold is projective and simply connected if the mixed curvature is positive.
\end{abstract}


\maketitle

\markboth{ Jianchun Chu, Man-Chun Lee and Luen-Fai Tam}{}

\section{Introduction}

In this work, we will study the structure of compact \K  manifolds $(M^n,h)$ with certain mixed    curvature conditions. More precisely, for any constants $\a, \b$, we consider the following mixed curvature defined on $T^{1,0}(M)$:
\be\label{e-C}
 \mathcal{C}_{\a,\b}(X)=\a \Ric(X,\ol X)+\b |X|^{-2}_h R(X,\ol X, X,\ol X)
\ee
for $X\in T^{1,0}(M)$ where $|X|_h\neq 0$. Here $R$ and $\Ric$ denote the curvature tensor and Ricci tensor of $h$ respectively.  In most cases, the dependency of $\mathcal{C}_{\a,\b}$ on $h$ is clear.    If we want to emphasis the dependency of $\mathcal{C}_{\a,\b}$ on $h$,  we will denote $\mathcal{C}_{\a,\b}$ by  $\mathcal{C}_{\a,\b}(h)$,   or $\mathcal{C}_{\a,\b,h}$ . The study is motivated by the following:
\begin{enumerate}
  \item [(i)]   $\mathcal{C}_{1,0,h}(X)  $ is the standard Ricci curvature.
  \item [(ii)] $\mathcal{C}_{0,1,h}(X)$ is the holomorphic sectional curvature up to scaling.
      \item[(iii)] $\mathcal{C}_{1, 1,h}(X)$  is the notion $\Ric^+ (X,\ol X)$   introduced by Ni \cite{Ni2019}. See Corollary \ref{c-simple-connected} for the definition.
          \item[(iv)]$\mathcal{C}_{1, -1,h}(X)$  is the orthogonal Ricci curvature  $\Ric^\perp (X,\ol X)$  introduced by Ni-Zheng \cite{NiZheng2018-1}, and is an analogue of the orthogonal bisectional curvature.
          \item[(v)]$\mathcal{C}_{k-1, n-k,h}(X)$ is closely related to the $k$-Ricci curvature introduced by Ni \cite{Ni2018},  Lemma \ref{l-Ric-k-1}.
\end{enumerate}

Motivated by the previous studies on different notions of curvature as above, we would like  to understand the structure of $(M,h)$ in which $\mathcal{C}_{\a,\b,h}\leq 0$, or more generally, for all non-zero $X\in T^{1,0}M$,
\be\label{e-negative}
 \mathcal{C}_{\a,\b}(X)+\ddb \phi(X,\ol X) \le \lambda  |X|_h^2
 \ee
for some smooth function $\phi$ and continuous function $\lambda$ which is nonpositive, or quasi-negative, or negative. We can also consider  its counter part:
there is a smooth function $\phi$ so that
\be\label{e-positive}
 \mathcal{C}_{\a,\b}(X)+\ddb \phi(X,\ol X)\ge \lambda|X|^2_h
 \ee
for some continuous function $\lambda$ which is positive.

In addition to the fact that $\mathcal{C}_{\a,\b}$ is a generalization of some known notions of curvature, it can also be defined in general for Hermitian metric if we replace the Ricci curvature by the first Ricci curvature of the Chern connection. Moreover, the above conditions are invariant under conformal changes (in the Hermitian category) in the sense that if $h$ satisfies \eqref{e-negative} for some $\phi$ and $\lambda$ with $\lambda\le 0$, then a conformal metric $e^{2f}h$ also satisfies the same condition with the same $\a, \b$ but with another $\wt \phi$ and $\wt \lambda$ which is still nonpositive, etc. We refer readers to Lemma~\ref{Lemma: Hermitian} for details.

Under the nonpositive condition \eqref{e-negative}, we have the following:
\begin{thm}\label{t-main-1}
Let  $(M^n,h)$ be a compact \K manifold satisfying \eqref{e-negative} for some constants $\a\geq 0, \b>0$, $\phi\in C^\infty(M)$ and a continuous   function $\lambda$.
\begin{enumerate}
  \item [{\bf(a)}] Suppose $\lambda\le0$,   then the canonical bundle $K_M$ of $M$ is numerically effective (nef).
  \item  [{\bf(b)}] Suppose $\lambda<0$, then $M$ is projective with ample $K_M$ and $M$ supports a \KE metric with negative scalar curvature.
  \item [{\bf(c)}] Suppose $\lambda$ is quasi-negative,  then $M$ is projective with $\int_M c_1(K_M)^n>0$.  In particular, together with {\bf(a)},  $K_M$ is big. If in addition $M$ does not contains any rational curve, then $K_M$ is ample.
\end{enumerate}
\end{thm}
 Recall that $K_M$ is numerically effective if for any $\e>0$, there is a smooth function $f$ so that $-\Ric(h)+\ddb f+\e \omega_h>0$ where $\omega_h$ is the \K form of $h$. $K_M$ is big if its Kodaira-Iitaka dimension of $K_M$ is equal to the dimension of $M$.

 In case of $\mathcal{C}_{1,0}$, the existence of \KE metric in Theorem \ref{t-main-1} {\bf(b)} is  the celebrated result of Aubin \cite{Aubin1976} and Yau \cite{Yau1978}. In the case  $\mathcal{C}_{0,1}$,  \eqref{e-negative} with $\phi=0$ is equivalent to the condition that  the holomorphic sectional curvature $H$ being bounded from above by $\lambda$.  In this case,  Theorem \ref{t-main-1}   was proved by  Wu-Yau  \cite{WuYau2016,WuYau2016-2}, Tosatti-Yang \cite{TosattiYang2017} and Diverio-Trapani \cite{DiverioTrapani2016}.  Note that  $H(X)\le 0$ implies that $M$ does not contain any rational curve by the work of Royden \cite{Royden1980}.  There have been   important  contributions on compact \K manifolds with $H\le 0$ by many other people, we refer readers to  \cite{HeierLuWong2010,HeierLuWong2016,
HeierLuWongZheng2017,LeeStreets2019,Nomura,YangZheng2016} and the reference therein for further details. Note that in our case, even for $\mathcal{C}_{0,1}$, we may allow $\phi$ to be nonzero in the condition \eqref{e-negative}.

As an application, Theorem \ref{t-main-1} gives an affirmative answer to a question of Ni \cite{Ni2018} asking whether a compact \K manifold with $\Ric_k<0$ is projective. On the other hand, in the Hermitian category, it was conjectured by Yang-Zheng \cite[Conjecture 1.1]{YangZheng2016} that  a compact Hermitian manifold with quasi-negative holomorphic sectional curvature is projective with ample $K_M$. Since condition \eqref{e-negative} is invariant under  conformal deformation, one can use Theorem~\ref{t-main-1}  to verify the conjecture is true for the special case that the Hermitian metric is conformally K\"ahler.

 We will use the twisted \KR flow \footnote{See also \cite{NiElliptic} for an elliptic proof which is based on the earlier version of this preprint.} to study the nefness and ampleness of $K_M$. Namely, we will study the following
\be\label{e-TKR-1}
\left\{
\begin{array}{ll}
\partial_t\omega(t)&= -\Ric(\omega(t))-\eta;\\
\omega(0)&=\omega_h,
\end{array}
\right.
\ee
  where $\omega(t)$ is a smooth family of \K forms, $\omega_h$ is the \K form of $h$, and  $\eta$ is a  smooth closed real $(1,1)$ form on $M$.  When $\eta=0$, this coincides with the usual \KR flow.  We will prove Theorem \ref{t-main-1}{\bf(a)} by showing the long-time existence.  We will prove Theorem \ref{t-main-1}{\bf(b)} by studying its long-time behaviour. The method may have independent interest. Finally, we will modify the continuity method of Wu-Yau \cite{WuYau2016-2} to obtain  Theorem \ref{t-main-1}{\bf(c)}.

 Next, let us discuss condition \eqref{e-positive}.
 There are many results on the structure of compact \K manifolds related to positive curvature recently. In \cite{Yang2018}, Yang proved a conjecture of Yau \cite{Yau1982} that a compact \K manifold with $\Ric_1>0$, i.e. positive holomorphic sectional curvature, must be projective and rationally connected, see also \cite{HeierWong2015}. The result is generalized further by Ni \cite{Ni2019} and Ni-Zheng \cite{NiZheng2018-1} recently. In particular, in \cite{Ni2019,NiZheng2018-1}, it was shown that compact \K manifolds with $\Ric^+>0$,  $\Ric^\perp>0$ or $\Ric_k>0$ are simply connected and projective. For more recent development, we refer interested readers to \cite{HeierWong2015,Matsumura2018-1,Matsumura2018-2,
Matsumura2018-3,NiZheng2018-1,NiZheng2018,
 Yang2018-1,Yang2020} and references therein.
%
 In this regard, following the argument in \cite[Theorem 2.7]{Ni2019}, we show that \K manifolds satisfying   \eqref{e-positive} with $\phi=0$  are also simply connected and projective.
\begin{thm}\label{t-positive-simplyconnect}
Suppose $(M,h)$ is a compact \K manifold with
\begin{equation}\label{positive-mixed-1}
\mathcal{C}_{\a,\b}(X)>0
\end{equation}
for some $\a, \b$ with $\a>0,  \a+\b\ge0$ and all non-zero $X\in T^{1,0}M$, then $h^{p,0}=0$ for all $1\leq p\leq n$. In particular, $M$ is simply connected and projective.
\end{thm}

The paper is organized as follows: In Section \ref{s-Royden},   using Royden's idea \cite{Royden1980} we show that $\Ric_k\le \lambda$  (respectively $\ge \lambda$) implies \eqref{e-negative} (respectively  \eqref{e-positive}) for some $\a, \b$, which in turns will give a useful  upper bound for the trace of the bisectional curvature. In Section \ref{C^0-estimate-KRF}, we will prove the nefness and ampleness of the canonical line bundle using the twisted \KR flow. In Section \ref{sec: quasi}, we will consider the quasi-negative case using compactness argument in \cite{WuYau2016-2} together with the twisted Monge-Amp\`ere equation. In Section~\ref{sec: Kahler}, we apply Theorem~\ref{t-main-1} to prove the structural theory of manifolds with non-positive $\Ric_k$ and $\Ric^+$. In Section~\ref{Sec: Hermitian}, we will show that \eqref{e-negative} is conformally invariant and a special case of \cite[Conjecture 1.1]{YangZheng2016}.  In Section~\ref{sec-positive}, we will prove Theorem~\ref{t-positive-simplyconnect}.\vskip .2cm

{\it Acknowledgement}: The authors would like to thank Lei Ni, Fangyang Zheng for their interest in this work and Shin-ichi Matsumura for sending us the related works on semi positive holomorphic sectional curvature. We also would like to thank Damin Wu and Shing-Tung Yau for useful discussions. Part of the works was done when the second  author visited the Institute of Mathematical Sciences at The Chinese University of Hong Kong, which he would like to thank for the hospitality.


\section{On symmetric bihermitian forms}\label{s-Royden}

In this section, we will first prove that the condition $\Ric_k\le \lambda$ will imply \eqref{e-negative} for some suitable $\a, \b$. This is essentially a result in linear algebra  on symmetric bihermitian form on a Hermitian vector space.  We will generalize  the result of Royden \cite{Royden1980}  where  an upper bound  on the trace of bisectional curvature  in terms of an upper bound of holomorphic sectional curvature was obtained.

Let $(V^n, h)$ be a Hermitian vector space with Hermitian metric $h$. Let $S(X,\ol Y, Z,\ol W)$ be a bihermitian form on $V^n$, namely $S(X,\ol Y,\cdot,\cdot)$ and $S(\cdot,\cdot, W,\ol Z)$ are Hermitian forms for fixed $X, Y$ or $W, Z$. Moreover, we assume that the bihermitian form $S$ satisfies the following symmetry:
\be\label{e-bihermitian}
S(X,\ol Y, Z,\ol W)=S(Z,\ol Y, X,\ol W), \ \
 \ol{S(X,\ol Y, Z,\ol W)}=S(Y,\ol X,W, \ol Z).
\ee
For $S$, we define its holomorphic sectional curvature, Ricci tensor and scalar curvature by
\[
H^S(X)=\frac{S(X,\ol X,X,\ol X)}{|X|^4_h}, \ \
\Ric^S(X,\ol Y)=\tr_h(S(X,\ol Y, \cdot,\cdot)), \ \
\cS=\tr_h \Ric^S.
\]
For any $k$-dimensional subspace $U$  of $V$, define the $\Ric_{k,U}^S$ to be the Ricci tensor for $S|_U$. Note that $\Ric_{k,U}^S$ depends also on the subspace $U$. We say that the $k$-Ricci curvature of $S$ is less than or equal to $\lambda$, denoted by $\Ric_k^S\leq \lambda$, for some $\lambda\in \mathbb{R}$ if for  any $X\in V$ and any $k$-dimensional subspace $U$ containing $X$, we have $\Ric_{k,U}^S(X,\bar X)\leq \lambda |X|^2_h$. The notion $\Ric_k^S\geq \lambda$ is defined analogously. It is easy to see that $\Ric_k^S$  on a unit vector is the holomorphic sectional curvature if $k=1$ and is the Ricci tensor for $V$ if $k=n$. We first give some relation between $\Ric_k$ and $\mathcal{C}_{\a,\b}$.


\begin{lma}\label{l-Ric-k-1}
Suppose $(V^n,h)$ is a Hermitian vector space and $S$ is a bihermitian form on $V$ satisfying \eqref{e-bihermitian}. For $1<k<n$, let $\{e_1,\dots, e_n\}$ be a unitary basis and for any $I \subset \{2,\dots,n\}$ with $|I|=k-1$, let $U_I$ be the subspace spanned by $\{e_1\}\cup \{e_i|i\in I\}$. Then
\bee
(k-1)\sum_{I; I \subset \{2,\dots,n\},|I|=k-1 } \Ric_{k,U_I}^S(e_1,\ol e_1)
= C_{k-2}^{n-2} \lf((n-k)H^S(e_1)+(k-1)\Ric^S(e_1,\ol e_1)\ri)
\eee
\end{lma}
\begin{proof}
For notational convenience, we use $S_{i\bar jk\bar l}$ to denote $S(e_i,\bar e_j,e_k,\bar e_l)$. Let $I$ be a  subset of $\{2,\dots,n\}$ such that $|I|=k-1$. Then there are $C^{n-1}_{k-1}$ different such $I$'s. For each $i\neq 1$, there are $C^{n-2}_{k-2}$ different such $I$'s which also contain $i$. Hence
\begin{equation*}
\begin{split}
\sum_{I; I \subset \{2,\dots,n\},|I|=k-1 } \Ric_{k,U_I}^S(e_1,\ol e_1)
 =&\sum_{I; I \subset \{2,\dots,n\},|I|=k-1 } \left(S_{1\bar1 1\bar 1}+\sum_{j\in I}S_{1\bar 1j\bar j} \right)\\
=&C^{n-1}_{k-1} S_{1\bar 11\bar 1}+C^{n-2}_{k-2} \sum_{i=2}^nS_{1\bar 1i\bar i}\\
=&\left(C^{n-1}_{k-1}-C^{n-2}_{k-2}\right) S_{1\bar 11\bar 1}+C^{n-2}_{k-2}{\sum_{i=1}^n S_{1\bar1i\bar i}}\\
=&C^{n-2}_{k-2}\left(\frac{n-k}{k-1}H^S(e_1)+{\Ric^S(e_1,\ol e_1)} \right).
\end{split}
\end{equation*}
From this, the result follows.
\end{proof}
\begin{lma}\label{Ric-k-interpolation}
Suppose $(V^n,h)$ is a Hermitian vector space and $S$ is a bihermitian form on $V$ satisfying \eqref{e-bihermitian}. Suppose $\Ric_k^S \leq  (k+1)\sigma$ for some integer $1\leq k\leq n$ and $\sigma\in \mathbb{R}$. Then for any $X\in V$, $X\neq 0$,  we have
$$ (k-1)  \Ric^S(X,\ol X)+(n-k)|X|^{-2}_hH^S(X)   \leq  {(n-1)(k+1)}\sigma |X|^2_h .$$

Similarly, if $\Ric_k^S \ge  (k+1)\sigma$, then
$$ (k-1)  \Ric^S(X,\ol X)+(n-k)|X|^{-2}_hH^S(X)   \ge  {(n-1)(k+1)}\sigma |X|^2_h .$$
\end{lma}
\begin{proof} We may assume $|X|_h=1$. If $k=1$ or $n$, the result is obviously true. In case $1<k<n$, we extend $e_1$ to be a unitary frame $\{e_1,\dots,e_n\}$. Suppose  $\Ric_k^S \leq  (k+1)\sigma$. With the notation as in Lemma~\ref{l-Ric-k-1}, we have
\bee
\begin{split}
(k-1)(k+1)C_{k-1}^{n-1}\sigma\ge&(k-1)\sum_{I; I \subset \{2,\dots,n\},|I|=k-1 } \Ric_{k,U_I}^S(e_1,\ol e_1)\\
=& C_{k-2}^{n-2} \lf((n-k)H^S(e_1)+(k-1)\Ric^S(e_1,\ol e_1)\ri).
\end{split}
\eee
From this, it is easy to see that the first part of the lemma is true. The case that
$\Ric_k^S \ge  (k+1)\sigma$ can be proved similarly.
\end{proof}

Now we adapt the trick of Royden \cite{Royden1980} to deduce an upper bound   on the trace of  bisectional curvature with respect to another Hermitian metric $g$. Motivated by Lemma \ref{Ric-k-interpolation}, we consider bihermitian form $S$ such that there is a real (1,1) form $\rho$ satisfying the following:
\be\label{e-Royden-1}
  h(X,\bar X)\cdot \rho(X,\bar X)+\b  S(X,\bar X,X,\bar X)\leq \lambda|X|_h^4
\ee
for all $X\in V^{1,0}$ for some constants $\b, \lambda$ where  $ \b$ are positive.

\begin{lma}\label{Ric-k-BK}
Suppose $(V^n,h)$ is a Hermitian vector space and $S$ is a bihermitian form on $V$ satisfying \eqref{e-bihermitian} and \eqref{e-Royden-1}. If $g$ is another Hermitian metric on $V$, then we have:
\bee
\begin{split}
 2g^{i\bar j}g^{k\bar l}  S_{i\bar jk\bar l}\le {} & \frac1{ \beta}\lf(\lambda (\tr_gh)^2- \tr_gh\cdot \tr_g\rho\ri)+\sum_{i=1}^nS(E_i,\overline{E_i},E_i,\overline{E_i})\\
 \le{}&\frac{\lambda}\b\lf((\tr_gh)^2+|h|_g^2\ri)-\frac1\b \tr_gh\cdot \tr_g\rho-\frac1\b\la\omega_h,\rho\ra_g
 \end{split}
\eee
where $\{E_i\}_{i=1}^n$ is a  unitary frame with respect to $g$ so that $h$ is diagonal. Here $S_{i\bar jk\bar l}=S(E_i,\bar E_j,E_k,\bar E_l)$ and $\omega_h$ is the \K form for $h$.
\end{lma}
\begin{proof} 
We follow closely the argument of Royden \cite{Royden1980}. Let $\{E_i\}_{i=1}^n$ be a frame such that $g(E_i,\overline{E_j})=\delta_{ij}$ and $h(E_i,\overline{ E_j})=\tau_i \delta_{ij}$. Let $\eta_A=\sum_{i=1}^n\e_i^A E_i$ where {$A= (\e_i^A)\in \mathbb{Z}_4^n$} with $\mathbb{Z}_4$ being the finite group consisting of $4$-th roots of unity. Note that for any $A\in \mathbb{Z}_4^n$, we have
\begin{equation}\label{trace-1}
h(\eta_A,\bar \eta_A) = \sum_{i,j=1}^n \e_i^A\overline{\e_{j}^A} h(E_i,\overline{ E_j})
=\sum_{i=1}^n \tau_i = \tr_gh.
\end{equation}
Also since $\sum_{A\in \mathbb{Z}_4^n}\e_i^A\overline{\e_j^A}=0$ for all $i\neq j$, by symmetry we obtain
\begin{equation}\label{trace-2}
\sum_{A\in \mathbb{Z}_4^n} \rho(\eta_A,\overline{\eta_A})
=\sum_{A\in \mathbb{Z}_4^n} \sum_{i,j=1}^n \e_i^A \overline{\e_j^A} \rho(E_i,\overline{E_j})
=4^n \sum_{i=1}^n \rho(E_i,\overline{E_i})
=4^n \tr_g \rho.
\end{equation}
Similarly,
\begin{equation*}
\begin{split}
& \sum_{A\in \mathbb{Z}_4^n}S(\eta_A,\overline{\eta_A},\eta_A,\overline{\eta_A})
=\sum_{A\in \mathbb{Z}_4^n} \sum_{i,j,k,l}^n \e_i^A\overline{\e_j^A}\e_k^A\overline{\e_l^A}S(E_i,\overline{E_j},E_k,\overline{E_l})\\
= {} & 4^n \sum_{i\neq j}\left(S(E_i,\overline{E_j},E_j,\overline{E_i})+S(E_i,\overline{E_i},E_j,\overline{E_j})\right)
+4^n \sum_{i=1}^nS(E_i,\overline{E_i},E_i,\overline{E_i})\\
= {} & 4^n g^{i\bar j}g^{k\bar l}\left( S_{i\bar jk\bar l} +S_{i\bar lk\bar j}\right)-4^n  \sum_{i=1}^nS(E_i,\overline{E_i},E_i,\overline{E_i})\\
= {} & 4^n\lf(2g^{i\bar j}g^{k\bar l}  S_{i\bar jk\bar l} -  \sum_{i=1}^nS(E_i,\overline{E_i},E_i,\overline{E_i})\ri)
\end{split}
\end{equation*}
and
\begin{equation*}
\begin{split}
&\sum_{A\in \mathbb{Z}_4^n}|\eta_A|^4_h
=\sum_{A\in \mathbb{Z}_4^n} \sum_{i,j,\gamma,\delta=1}^n \e_i^A\overline{\e_j^A}\e_\gamma^A\overline{\e_\delta^A}h(E_i,\overline{E_j})h(E_\gamma,\overline{E_\delta})\\
={}&4^n \sum_{i\neq j} \left(h(E_i,\overline{E_j})h(E_j,\overline{E_i})+h(E_i,\overline{E_i})h(E_j,\overline{E_j})\right)
+4^n \sum_{i=1}^nh(E_i,\overline{E_i})h(E_i,\overline{E_i})\\
={}&4^n\lf(\sum_{i\neq j}\tau_i\tau_j+\sum_i\tau_i^2\ri)
=4^n(\tr_gh)^2.
\end{split}
\end{equation*}
Applying \eqref{e-Royden-1} for each $\eta_A$ and summing over $A\in \mathbb{Z}_4^n$, we conclude that
\bee
\begin{split}
\lambda (\tr_gh)^2\ge  \tr_gh\cdot \tr_g\rho+\beta\lf(2g^{i\bar j}g^{k\bar l}  S_{i\bar jk\bar l} -  \sum_{i=1}^nS(E_i,\overline{E_i},E_i,\overline{E_i})\ri).
\end{split}
\eee
This implies the first inequality in the Lemma. The second follows from \eqref{e-Royden-1}.
\end{proof}


\section{Canonical line bundle under non-positive curvature}\label{C^0-estimate-KRF}

In this section, we are going to prove Theorem \ref{t-main-1} {\bf(a)}, {\bf(b)}.
Consider a compact \K manifold $(M,h)$ with curvature $R_h$ satisfying \eqref{e-negative}. Hence it satisfies \eqref{e-Royden-1} for $\rho=\a\Ric_{h}+\ddb   \phi$. That is:
\be\label{e-condition-1}
h(X,\ol X)\cdot  (\a \Ric(h)+\ddb \phi)(X,\ol X)+\b  R^h(X,\ol X,X,\ol X)\le \lambda |X|^4
\ee
for some $ \a\geq 0, \b>0$ and for some function $\phi\in C^\infty(M)$ and non-positive function $\lambda$. We first show that $K_M$ is nef using \eqref{e-TKR-1}.  This flow is equivalent to the following Monge-Amp\`ere type flow:
\begin{equation}\label{e-TKR-2}
\left\{
\begin{array}{ll}
\partial_t\varphi &=\displaystyle{\log\frac{(\omega_h-t\Ric(\omega_h)-t\eta+\ddb\varphi)^n}{\omega_h^n}}; \\
\varphi(0)&=0,
\end{array}
\right.
\end{equation}
in the sense that if $\varphi$ satisfies \eqref{e-TKR-2} on $M\times[0,T]$ so that
\begin{equation}\label{TKRF-form}
\omega_h-t\Ric(\omega_h)-t\eta+\ddb\varphi>0,
\end{equation}
then $\omega(t)=\omega_h-t\Ric(\omega_h)-t\eta+\ddb\varphi$ will satisfy \eqref{e-TKR-1}. Moreover if $\omega(t)$ satisfies \eqref{e-TKR-1}, then $$\varphi(t)=\int^t_0\log \frac{\omega(s)^n}{\omega_h^n}\, ds$$ satisfies \eqref{e-TKR-2}.  In the following, $\varphi$ will always denote the solution of \eqref{e-TKR-2} which corresponds to \eqref{e-TKR-1}.  Moreover, $g(t)$ will be the \K metric with respect to the \K form $\omega(t)$.

Since $M$ is closed,  the twisted \KR flow $\omega(t)$ admits a short time solution, for example see \cite{GuedjZeriahi2017}.
In this section, we will estimate the existence time of the flow $g(t)$ under the assumption \eqref{e-condition-1}. We need the following fact, which states that if the solution $g(t)$ to \eqref{e-TKR-1} is uniformly equivalent to a fixed metric $h$ on $M\times[0,T_0)$, then we have higher order regularity of $g(t)$ up to $T_0$ and hence the solution can be extended beyond $T_0$, see \cite{Chu2016,Evans1982,Krylov1982,
ShermanWeinkove2012,TosattiWangWeinkoveYang}.  To summarize, we have the following existence time criteria for the parabolic complex  Monge-Amp\`ere equation.
\begin{lma} \label{higher-order}Let $g(t)$ be a smooth solution to \eqref{e-TKR-1} on $M\times[0,T_0)$. Suppose there is a positive constant $C>0$ such that
$$
C^{-1}h\le g(t)\le Ch
$$
on $M\times[0,T_0)$. Then there is $\e>0$ such that $g(t)$ can be extended to $M\times[0,T_0+\e)$ which satisfies \eqref{e-TKR-2}.
\end{lma}
\begin{proof}  We sketch the proof as follows. The condition that $C^{-1}h\le g(t)\le Ch$ ensure that we have uniform $C^0$ estimate. up on $[0,T_0)$  If $\eta=0$, \eqref{e-TKR-1} is the ordinary  \KR flow. By \cite[Corollary 1.1]{ShermanWeinkove2012},   all the derivatives of $g(t)$ with respect  to $h$ are uniformly bounded on $[\frac {T_0} 2, T_0)$. One can combine the method in \cite{ShermanWeinkove2012} with the techniques for more general fully nonlinear elliptic equations as in \cite{Chu2016,Evans1982,Krylov1982,TosattiWangWeinkoveYang} to show that the same estimates are true for general $\eta$. From this it is easy to see that the solution $g(t)$ can be extended beyond $T_0$.
\end{proof}

By direct calculation, we have the following well-known formulas,  for example see \cite{TosattiNote} and \cite{KRFnote}.
\begin{lma}\label{l-phi}
Let $\dot \varphi=\partial_t\varphi$. We have
\bee
\left\{
  \begin{array}{ll}
   \displaystyle{ \heat \dot\varphi=-\tr_g(\Ric(h)+\eta);}\\
   \displaystyle{\heat (t\dot\varphi-\varphi-nt)=-\tr_g h.}
  \end{array}
\right.
\eee
\end{lma}
\begin{proof}
 The formulas are standard, we include the proof for reader's convenience. By differentiating \eqref{e-TKR-2}  with respect to time, we have
\begin{equation}
\begin{split}
\partial_t\dot\varphi&=\tr_g\lf(\frac{\p}{\p t} \lf(\omega_h-t\Ric( h)-t\eta+\ddb \varphi\ri)\ri)
\\&=\tr_g\left(-\Ric(h)-\eta+\ddb \dot\varphi \right)\\
&=-\tr_g (\Ric(h)+\eta)+\Delta_{g(t)} \dot\varphi.
\end{split}
\end{equation}
This proved the first equation. For the second equation,
 taking trace with respect to $g(t)$ of the equality
$$\omega(t)=\omega_h-t\Ric(\omega_h)-t\eta+\ddb\varphi,$$
 we have
$-t\tr_g(\Ric(h)+\eta)=n-\Delta_g\varphi-\tr_g(h)$. Hence
\begin{equation}
\begin{split}
\heat (t\dot\varphi)&=-t\tr_g ( \Ric(h)+\eta )+\dot\varphi\\
&=n-\tr_gh+\heat \varphi.
\end{split}
\end{equation}
This proved the second equation.
\end{proof}

Along the \KR flow, it is well-known that the scalar curvature is bounded from below. We have the following analogy for the twisted \KR flow.
\begin{lma}\label{volume-upper}
Let $g(t),t\in [0,T)$ be a solution to the  twisted \KR flow on $M$ with initial metric $g(0)=h$. Then the scalar curvature $\mathcal{S}(g(t))$ satisfies
$$\mathcal{S}(g(t))+\tr_g\eta \geq -\frac{n}{t+\sigma}$$
on $M\times [0,T)$ where $\sigma>0$ so that $\inf_M ( \mathcal{S}(h)+\tr_{h}\eta)\geq -n\sigma^{-1}$.  In particular,
\begin{equation}
 \sup_M\left(\log \frac{\det g(t)}{\det h}\right)=\sup_M\dot\varphi(\cdot,t)\leq \displaystyle n\log  \left(\frac{t+\sigma}{\sigma}\right).
\end{equation}
\end{lma}
\begin{proof}
Direct calculation shows
\begin{equation*}
\begin{split}
  \partial_t (\mathcal{S}+\tr_g\eta)=&-g^{i\bar l}g^{k\bar j}(R_{i\bar j}+\eta_{i\bar j})\partial_t g_{k\bar l}+g^{i\bar j}\partial_t R_{i\bar j}\\[1mm]
= {} &|\Ric+\eta|^2-g^{i\bar j}\partial_i\partial_{\bar j}\left(g^{k\bar l}\partial_t g_{k\bar l}\right)\\
=&|\Ric+\eta|^2+\Delta_{g(t)} (\mathcal{S}+\tr_g\eta)\\
\ge {} & \frac1n(\mathcal{S}+\tr_g\eta)^2+ \Delta_{g(t)} (\mathcal{S}+\tr_g\eta).
\end{split}
\end{equation*}
The lower bound of $\mathcal{S}+\tr_g\eta$ follows from the maximum principle. The upper bound of $\dot\varphi$ follows from the fact that $\partial_t\dot\varphi=-\mathcal{S}-\tr_g\eta$ and $\dot\varphi(0)=0$.
\end{proof}
The lemma gives an upper bound of $ (\det g)/(\det h) $. Next, we want to estimate the upper bound of $\tr_gh$. Combining these two bounds, one can obtain $C^0$ estimates along the twisted \KR flow. This in turns will give an estimate on the existence time. We need the following parabolic Schwarz Lemma which  is a straight forward application of Yau's Schwarz Lemma \cite{Yau1978-2}.

\begin{lma}\label{l-Schwarz}
Let $g(t)$ be a solution to the twisted \KR flow \eqref{e-TKR-1}, then
\bee
\begin{split}
\heat \log \tr_gh \leq& \frac{1}{\tr_gh}g^{i\bar j}g^{k\bar l}  R_{i\bar jk\bar l}(h)+\frac{1}{\tr_gh}g^{i\bar l}g^{k\bar j}h_{i\bar j}\eta_{k\bar l}.
\end{split}
\eee
\end{lma}
\begin{proof}
The proof is similar to the parabolic Schwarz Lemma in \KR flow \cite{SongTian2007},  see also \cite[Theorem 2.6]{KRFnote}. By Yau's Schwarz Lemma \cite{Yau1978-2}, for two \K metrics $g$ and $h$ we have
\begin{equation}\label{sch}
\Delta_{g}\log \tr_gh \geq \frac{1}{\tr_gh} \left(R^{i\bar j}(g)h_{i\bar j}-g^{i\bar j}g^{k\bar l}R_{i\bar jk\bar l}(h)\right).
\end{equation}
Applying \eqref{sch} with $g=g(t)$, we conclude that
\begin{equation}
\begin{split}
&\quad \heat \log \tr_gh\\
&\leq \frac{1}{\tr_gh} \left(g^{i\bar j}g^{k\bar l}R_{i\bar jk\bar l}(h)-R^{i\bar j}(g)h_{i\bar j}\right)+\frac{1}{\tr_gh} h_{i\bar j} \partial_t g^{i\bar j}\\
&= \frac{1}{\tr_gh}g^{i\bar j}g^{k\bar l}  R_{i\bar jk\bar l}(h)+\frac{1}{\tr_gh}g^{i\bar l}g^{k\bar j}h_{i\bar j}\eta_{k\bar l},
\end{split}
\end{equation}
 where we have used \eqref{e-TKR-1} in the last step.
\end{proof}

We are now ready to prove Theorem \ref{t-main-1} {\bf(a)}.
\begin{proof} [Proof of Theorem \ref{t-main-1} {\bf(a)}]
 For the \K metric $h$, we let
$$S=\inf\{s\in \mathbb{R}: \exists f\in C^\infty(M), \Ric(h)< s\omega_h+\ddb f\}.$$
We claim that $S\le 0$. If the claim is true, then for any $\e>0$ we can find smooth function $f$ so that $-\Ric(h)-\ddb  f\ge-\e \omega_h$. Since the first Chern class of the canonical line bundle $K_M$ is represented by $-\Ric(h)$, we see that the canonical bundle is nef.

Suppose on the contrary  that $S>0$.  Since $\a\geq 0$, for $\mu>S$ to be chosen later,  we can find $v\in C^\infty(M)$ such that
\begin{equation}\label{contra-assumpt}
\a\Ric(h)+\ddb \phi=\rho \leq \mu \a  \omega_h+\ddb v.
\end{equation}

Let $g(t)$ be the twisted \KR flow with $\eta=\ddb u$ and $g(0)=h$ where $ u=(2\b)^{-1}v$. We want to show that the maximal existence time $T_{\max}>S^{-1} $. If this is true, then \eqref{TKRF-form} implies that
$$
\omega(t)=\omega_h-t\Ric(h)-t\ddb u+\ddb \varphi
$$
is  a \K metric for some $t>S^{-1}$ where $\varphi$ is the solution to \eqref{e-TKR-2}. But this contradicts the definition of $S$. See also \cite{TianZhang2006,Tsuji1988} for the existence time characterization of the \KR flow.

 On $M\times[0,T_{\max})$, we are going to estimate $\Lambda:=\tr_gh$. Let $E_i$ be a  unitary frame with respect to $g(t)$ which diagonalizes $h$ at a point.
By Lemma \ref{l-Schwarz} and Lemma~\ref{Ric-k-BK} with $g=g(t)$ and $\rho=\a\Ric(h)+\ddb \phi$, we have
\begin{equation}\label{Schwarz-Lemma}
\begin{split}
 &\heat \log \Lambda
 \\
 \leq &\frac{1}{\Lambda}g^{i\bar j}g^{k\bar l}  R_{i\bar jk\bar l}(h)+\frac{1}{\Lambda}g^{i\bar l}g^{k\bar j}h_{i\bar j}u_{k\bar l}\\
\leq {} & \frac{\lambda}{2\b}\Lambda+\frac{\lambda}{2\b\Lambda}|h|^2
+\frac{1}{\Lambda}\langle \omega_h,\ddb u\rangle-\frac{1}{2 \beta}\tr_g\rho-\frac1{2\b\Lambda} \langle \rho,\omega_h\rangle\\
\leq {} & \frac{\lambda}{2\b}\Lambda+\frac{\lambda}{2\b\Lambda}|h|^2
-\frac{1}{ \beta}\tr_g\rho+\frac{1}{\Lambda}\langle \omega_h,\ddb u\rangle+\frac1{2\b\Lambda}\lf(\Lambda\tr_g\rho-\langle \rho,\omega_h\rangle\ri)\\
= {} & \frac{\lambda}{2\b}\Lambda+\frac{\lambda}{2\b\Lambda}|h|^2-{\frac{1}{ \beta}\tr_g \rho}+\frac{1}{\Lambda}\langle \omega_h,\ddb u\rangle
+\frac{1}{2\b\Lambda}\sum_{i=1}^n\rho(E_i,\bar E_i)\left( \Lambda-h(E_i,\bar E_i)\right)\\
\leq {} & \frac{\lambda}{2\b}\Lambda+\frac{\lambda}{2\b\Lambda}|h|^2-\frac{1}{ \beta}\tr_g\rho +\frac{1}{\Lambda}\langle \omega_h,\ddb u\rangle\\
&+\frac{1}{2\b\Lambda}\sum_{i=1}^n \left(\a\mu h(E_i,\bar E_i)+(\ddb v)(E_i,\bar E_i) \right)\left( \Lambda-h(E_i,\bar E_i)\right)\\
= {} &\left(\frac{\lambda+\a \mu}{2\b} \right)\Lambda+\left( \frac{\lambda-\a\mu}{2\b}\right)\frac{|h|^2}{\Lambda}+\frac{1}{2\b}\Delta_{g(t)}v-\frac{1}{ \beta}\tr_g\rho.
\end{split}
\end{equation}
Here we have used the fact that:
\bee
\frac1{2\b}\sum_{i=1}^n(\ddb v)(E_i,\bar E_i)h(E_i,\ol E_i)=
\la \ddb u,\omega_h\ra.
\eee
Since $n|h|^2\geq \Lambda^2$, then we have
\begin{equation*}
\left(\frac{\lambda+\a \mu}{2\b} \right)\Lambda+\left( \frac{\lambda-\a\mu}{2\b}\right)\frac{|h|^2}{\Lambda}\leq \frac{(n+1)\lambda+\a\mu (n-1) }{2n\b}\Lambda.
\end{equation*}
Let $w=t\dot\varphi-\varphi-nt$ and $B= \frac{(n+1)\lambda+\a\mu (n-1) }{2n\b}$. Combining the above with Lemma \ref{l-phi}, we deduce that
\be\label{evo-TKRF-1}
\begin{split}
& \heat \log\Lambda\le  B\Lambda+\frac{1}{2\b}\Delta_{g(t)}v-\frac{1}{ \beta}\tr_g\rho\\
= {} &-B\heat w +\frac{1}{2\b}\Delta_{g(t)}v+\frac1{\b}\heat  (\a\dot\varphi+\phi-\a u)\\
= {} &\heat \lf[-Bw- \frac1{2\b}v+\frac1\b (\a\dot\varphi+\phi-\a u)\ri].
\end{split}
\ee

Since at $t=0$,
$$
\log \Lambda+Bw+\frac1{2\b}v+\frac1\b (\a\dot\varphi -\phi+\a u)\le C
$$
for some constant $C$ depending only on $n, \a, \b, \sup_M|u|, \sup_M|\phi|$ and the upper bound of $\mu$.
It follows  from the maximum principle that
\begin{equation}\label{evo-TKRF-2}
\begin{split}
\log \Lambda\le & C_1+ \left(\frac\a\b -Bt\right)\dot \varphi+B\varphi +Bn t,
\end{split}
\end{equation}
where $C_1$ depends only on $n$, $\a$, $\b$, $\sup_M|v|$, $\sup_M |\phi|$ and the upper bound of $\mu$. Since $\lambda\leq0$,   $\left(\frac\a\b -Bt\right)\geq0$ for $t< \min\{ T_{\max},\frac{2n}{(n-1)\mu}\}$. Combining this with \eqref{evo-TKRF-2} and Lemma \ref{volume-upper}, we conclude that there exists a constant $C_2>1$ such that for all $t< \min\{ T_{\max},\frac{2n}{(n-1)\mu}\}$,
$$
C_2^{-1}h\le g(t)\le C_2 h.
$$
By Lemma~\ref{higher-order}, we conclude that $T_{\max}\ge \frac{2n}{(n-1)\mu}$. Hence $T_{\max}>S^{-1}$ if we choose $\mu$ sufficiently close to $S$. This completes the proof.
\end{proof}
Next we want to prove that $K_M$ is ample if $\lambda<0$.
\begin{proof}[Proof of Theorem \ref{t-main-1} {\bf(b)}]
By part {\bf (a)} of the theorem,  the canonical line bundle $K_M$ is nef. Therefore, for all $\e>0$, we can find $v\in C^\infty(M)$ such that
\begin{equation}
\label{nef-assumption}
 \a\Ric(h)+\ddb \phi=\rho\leq \e\a\omega_h+\ddb v.
\end{equation}
Choose $\e,\sigma>0$ small enough so that
$$B=\frac{(n+1)\lambda+(n-1)\e\a}{2\b n}\le -\sigma.
$$

Let $g(t)$ be the twisted \KR flow with $\eta=\ddb u$ where $ u=(2\b)^{-1}v$. By Theorem \ref{t-main-1} {\bf(a)} and the existence time estimates in \cite{TianZhang2006}, the flow exists for all time. Let $\Lambda=\tr_gh$ and $ F=\log \Lambda+\a\left(1+\b^{-1} \right)u-\b^{-1} (\a\dot\varphi+\phi)$. By Lemma \ref{l-phi} and the computation in \eqref{evo-TKRF-1},   we have
\begin{equation}
\heat F\le -\sigma \Lambda
\end{equation}
for some $\sigma>0$. Now let $G=F+\left(1+\frac{ \a n}{\b}\right)\log t$. Then
\begin{equation}\label{proof of b eqn}
\heat G\le -\sigma \Lambda+ \left(1+\frac{\a n}{\b}\right)t^{-1}.
\end{equation}
For any $T_0>0$, suppose that $G(x_0,t_0)=\sup_{M\times[0,T_0]}G$. Since $G\to-\infty$ as $t\to 0$, then $t_0>0$. At $(x_0,t_0)$, \eqref{proof of b eqn} shows $t_0\Lambda(x_0,t_0)\le C_1$
for some $C_1>0$ depending only on $\a, \b, \sigma$ and $n$. Hence by AM-GM inequality,
\begin{equation*}
\begin{split}
  \sup_{M\times [0,T_0]}G(x,t) \leq& G(x_0,t_0)\\
=   & \lf(\log \Lambda-\frac1\b{ (\a\dot\varphi+\phi)}+\a\left( 1+\frac1\b\right)u+\left(1+\frac{\a n}{\b}\right)\log t \ri)\Big|_{(x_0,t_0)}\\
= {} & \lf(\log \Lambda+\frac{\a n}{\b}\cdot \log \left(\frac{\det h}{\det g}\right)^{1/n}+\left(1+\frac{\a n}{\b}\right)\log t\right)\Big|_{(x_0,t_0)}
+C_2\\
\le {} & \left(1+\frac{\a n}{\b}\ri)\left(  \log \Lambda +\log t\right)\Big|_{(x_0,t_0)}+C_2\\
\leq &\lf(1+\frac{\a n} {\b}\ri) \log C_1  +C_2 =:C_3
\end{split}
\end{equation*}
for some constant $C_3>0$ independent of $T_0$. {By letting $T_0\rightarrow +\infty$,} we have an uniform upper bound of $\sup_{M\times [0,+\infty)}G$ and hence for $t$ large enough,
\bee
\begin{split}
\Lambda \leq C_4t^{-1}
\end{split}
\eee
for some constant $C_4$ independent of $t$ 
where we have used Lemma \ref{volume-upper}. Therefore, $g(t)\ge  {C_4}t h$ for some $C_4>0$ for $t$ large enough.  With the uniform lower bound, the normalized \KR flow will converge to the unique negative \KE metric, see \cite{HuangLeeTamTong2018} for example. One can also argue as follows. By rewriting $g(t)$ using potential, the lower bound implies that for sufficiently large $t$,
$$
-\Ric(h)-\ddb f\ge \e\omega_h
$$
for some $\e>0$ and $f(t)\in C^\infty(M)$.   By the result of Aubin \cite{Aubin1976} and Yau \cite{Yau1978}, $M$ supports a \KE metric with negative Ricci curvature. In particular, the canonical line bundle $K_M$ is ample and $M$ is projective.
\end{proof}

\section{Quasi-negative case}\label{sec: quasi}
In this section,   we will prove Theorem \ref{t-main-1} {\bf(c)}.   We consider compact \K manifolds which satisfies \eqref{e-condition-1} for some quasi-negative function $\lambda$. 
We will follow closely the arguments in \cite{WuYau2016-2}. Our main contribution is the following:
\begin{lma}\label{l-big} Let $(M^n,h)$ be a compact \K manifold satisfying \eqref{e-condition-1} for some quasi-negative function $\lambda$. Then $\int_M c_1(K_M)^n>0$.
\end{lma}
\begin{proof} It is equivalent to prove $\int_M (-\Ric(h))^n>0$. Since $K_M$ is nef by Theorem \ref{t-main-1} {\bf(a)}, for all $1\ge\e>0$ there exists $u_\e\in C^\infty(M)$ such that
\begin{equation}\label{equ-from-nef}
 -\Ric(h)+ \e\omega_h+\ddb u_\e>0.
\end{equation}
By Yau's Theorem \cite[Theorem 4, p.383]{Yau1978}, we can find $v_\e\in C^\infty(M)$ such that
\begin{equation}\label{seq-MA}
\left\{
\begin{array}{ll}
\left( -\Ric(h)+ \e\omega_h+\ddb (u_\e+v_\e)\right)^n=\exp\left(v_\e+u_\e+ \frac{1}{2\b}(\phi+\a u_\e)\right) \omega_h^n;\\
-\Ric(h)+ \e\omega_h+\ddb (u_\e+v_\e)>0.
\end{array}
\right.
\end{equation}
For notational convenience, we denote
$$ f_\e:=v_\e+u_\e+ \frac{1}{2\b}(\phi+\a u_\e); \ \ \omega_\e:=-\Ric(h)+ \e\omega_h+\ddb (u_\e+v_\e)$$
 so that
\begin{equation}
\omega_\e^n =\exp(f_\e)\,\omega_h^n; \  \ -\Ric(\omega_\e)+\Ric(\omega_h)=\ddb f_\e.
\end{equation}
By Stokes' theorem,
\begin{equation}
\lim_{\e\rightarrow 0}\int_M \omega_\e^n =\lim_{\e\rightarrow 0}\int_M \left(-\Ric(h)+\e \omega_h \right)^n= \int_M \left(-\Ric(h)\right)^n.
\end{equation}
Therefore, it suffices to estimate the lower bound of $\displaystyle \int_M \omega_{\e_i}^n$ for some $\e_i\to 0$.

In order to use Wu-Yau's method \cite{WuYau2016-2}, we will derive a differential inequality involving  of $\Lambda=\tr_{g}h$ where $g$ is the \K metric associated to the \K form $\omega_\e$. By Yau's Schwarz Lemma \cite{Yau1978-2},
\be\label{e-quasi-1}
-\Delta_{g}\log\Lambda \leq \frac1{\Lambda} g^{i\bar j}g^{k\bar l} R(h)_{i\bar jk\bar l}-\frac{1}{\Lambda}\la \Ric(g), h\ra.
\ee
Here and below the inner product is taken with respect to $g$. We choose a unitary frame $E_i$ with respect to $g$ so that $\rho$ is diagonal at a point. By \eqref{equ-from-nef} and similar computation of \eqref{Schwarz-Lemma}, for $\rho=\a\Ric(h)+\ddb \phi$, we see that
\begin{equation}
\begin{split}
 \frac1{\Lambda} g^{i\bar j}g^{k\bar l} R(h)_{i\bar jk\bar l}
\leq  {} & \frac{\lambda+\a\e}{2\b} \Lambda +\frac{\lambda-\a\e}{2\b \Lambda} |h|^2-\frac{\a}{\b}\tr_g \Ric(h)\\
&+\frac{1}{2\b}\Delta_g(\a u_\e-\phi)
-\frac1{2\b\Lambda}\la \ddb(\a u_\e+\phi),\omega_h\ra.
\end{split}
\end{equation}
On the other hand,
\be\label{e-quasi-3}
\begin{split}
-\frac{1}{\Lambda}\la \Ric(g), h\ra={}&-\frac{1}{\Lambda}\la \Ric(h)-\ddb f_\e, \omega_h\ra\\
={}&\frac1{2\b\Lambda}\la \ddb(\a u_\e+\phi),\omega_h\ra+\frac1\Lambda\la \omega_\e-\e\omega_h,\omega_h\ra
\end{split}
\ee
because $\ddb f_\e=\frac1{2\b}\ddb(\a u_\e+\phi)+\omega_\e-\e\omega_h+\Ric(h)$. By \eqref{e-quasi-1}--\eqref{e-quasi-3},  we conclude that
\begin{equation}
\begin{split}
-\Delta_{g}\log\Lambda &\le \left(\frac{\lambda+\e \a}{2\b} \right)\Lambda +\left(\frac{\lambda-\e \a}{2\b\Lambda} \right)|h|^2+\frac1{2\b}\Delta_g (\a u_\e-\phi)\\
&\quad -\frac\a\b\tr_g \Ric(h) +\frac1\Lambda \langle\omega_\e-\e\omega_h, \omega_h\rangle.
\end{split}
\end{equation}
Since $\lambda\leq 0$, by rewriting $-\Ric(h)=\omega_\e-\e\omega_h-\ddb (u_\e+v_\e)$, we can see that the function $ F=-\log \Lambda -\frac{1}{2\b}(\a u_\e-\phi)+\frac\a\b (u_\e+v_\e)$ satisfies
\begin{equation}\label{e-quasi-4}
\Delta_g F \leq \left(1+\frac{n\a}{\b} \right)+\left(\frac{\lambda}{2\b} \right)\Lambda
\leq  \left(1+\frac{n\a}{\b} \right)+\left(\frac{n\lambda}{2\b} \right)\exp\left( -\frac{\max_M f_\e}{n}\right).
\end{equation}
Here we have used AM-GM inequality at the last inequality. From this, one can proceed as in the proof of  \cite[Theorem 2]{WuYau2016-2}. We sketch the arguments for the convenience of the readers. First one can estimate $\sup_M f_\e$ as follows.   Since $M$ is compact, for all $1\ge \e>0$, there is $x_0\in M$ such that $f_\e(x_0)=\max_M f_\e$. At $x_0$, we have $\ddb f_\e(x_0) \leq 0$ and hence
\begin{equation}
\begin{split}
  \ddb (u_\e+v_\e) \leq &-\frac{1}{2\b}\ddb (\phi+\a u_\e)\\
\leq {} & -\frac{1}{2\b}\ddb \phi +\frac{\a}{2\b} \left(-\Ric(h)+\e\omega_h \right)\\
 \leq& C_0(\a,\b,\phi,n,h) \omega_h
\end{split}
\end{equation}
where we have used \eqref{equ-from-nef} in the second inequality. By substituting it back to the Monge-Amp\`ere equation \eqref{seq-MA}, we conclude that
\be\label{e-upperbound}
\sup_M f_\e\le C_1 (\a,\b,\phi,n,h),
\ee
which is independent of $\e$.
On the other hand, from \eqref{equ-from-nef} and \eqref{seq-MA}, we see that $C_1\omega_h+\ddb f_\e>0$ for some sufficiently large $C_1$ independent of $\e$. By \cite[Lemma 7]{WuYau2016-2}, we can find $\e_i\to0$ such that   the sequence $\{\exp(1+\sup_M f_{\e_i}-f_{\e_i})\}_{i=1}^\infty$   converges to $\exp(e^w)$ for some function $w$ almost everywhere. Since $(1+\sup_{M}f_{\e_i}-f_{\e_i})\geq1$, integrating \eqref{e-quasi-4} and using Lebesgue dominated convergence theorem, we obtain
\bee
\exp\left( -\frac{\max_M f_{\e_i}}{n}\ri)\le   \left(1+\frac{n\a}{\b} \right)\frac{\int_M\omega_{\e_i}^n}{\int_M \frac{-n\lambda}{2\b} \omega_{\e_i}^n}
\to  \left(1+\frac{n\a}{\b} \right)\frac{\int_M\exp(-e^w)\omega_h^n}{\int_M \frac{-n\lambda}{2\b} \exp(-e^w)\omega_h^n}
\eee
and so $\sup_M f_{\e_i}\ge -C_3$ for some $C_3>0$ independent of $i$. Together with the upper bound \eqref{e-upperbound},  passing to a subsequence $f_{\e_i}\to -e^w+c$ for some constant $c$. This implies that
$$
\int_M\omega_{\e_i}^n\to \int_M \exp(-e^w+c)\omega_h^n>0.
$$
This completes the proof.
\end{proof}


\begin{proof}[Proof of Theorem \ref{t-main-1} {\bf (c)}]
 By part {\bf(a)} of Theorem~\ref{t-main-1},  $K_M$ is nef.  Combine it with Lemma \ref{l-big},  it follows that $K_M$ is big  using \cite[Corollary 2.3.38, p.114]{MM07}. Hence  $M$ is Moishezon  by \cite[Theorem 0.5]{DemaillyPaun2004}. Since $M$ is K\"ahler, $M$ is projective by Moishezon's Theorem. If in addition that $M$ does not contain any rational curve, then $K_M$ is ample by the proof of \cite[Lemma 5]{WuYau2016}, see also \cite{DiverioTrapani2016}.
\end{proof}

\section{Applications on \K manifolds with non-positive curvature}\label{sec: Kahler}

Now we are in a position to apply Theorem~\ref{t-main-1} to prove the following structural result which answer a question of Ni \cite{Ni2019}.

\begin{thm}\label{t-main-2}
Suppose $(M^n,h)$ is a compact \K manifold with $\Ric_k(h)\leq -(k+1)\sigma$ for some non-negative continuous function $\sigma$ and integer $k$ with $1\leq k\leq  n$.
\begin{enumerate}
  \item [{\bf(a)}] The canonical bundle $K_M$ of $M$ is numerically effective (nef).
  \item  [{\bf(b)}] Suppose $\sigma>0$, then $K_M$ is ample. In particular, $M$ supports a \KE metric negative Ricci curvature and $M$ is projective.
  \item [{\bf(c)}] Suppose $\sigma$ is quasi-positive, then $\int_M c_1(K_M)^n>0$.
  In particular, $K_M$ is big and $M$ is projective. If in addition $M$ does not contains any rational curve, then $K_M$ is ample.
  \end{enumerate}
\end{thm}

\begin{proof}
If $k=1$ or $n$, the result is well-known. It suffices to consider $1<k<n$. By Lemma~\ref{Ric-k-interpolation}, the curvature of $g$ satisfies \eqref{e-condition-1} with $\a=k-1$, $\b=n-k$ and $\lambda=-(n-1)(k+1)\sigma$. The result follows from Theorem~\ref{t-main-1}.
\end{proof}


The condition \eqref{e-Royden-1} is also related to curvature $\Ric^+$ introduced in \cite{Ni2019} which is defined to be
$$
\Ric^+(X,\ol X)=\Ric(X,\ol X)+\frac{  R (X,\ol X,X,\ol X)}{|X|^2}.
$$
and is equivalent to the left hand side of \eqref{e-negative} with $\a=\b=1$ and $\phi=0$. It was proved in \cite[Proposition 6.2]{Ni2019} that a compact \K manifold with $\Ric^+<0$ has no nontrivial holomorphic vector field.
By Theorem~\ref{t-main-1}, we have the following  stronger results.
\begin{cor}
Suppose $(M^n,g)$ is a compact \K manifold with $\Ric^+\leq -(n+2)\sigma$ for some nonnegative function $\sigma$, then the canonical line bundle $K_M$ is nef. If $\sigma>0$ on $M$, then $K_M$ is ample.  If   $\sigma$ is quasi-positive, then $M$ is projective with $K_M$ being big and nef.
\end{cor}

\section{Conformal invariant of curvature condition}\label{Sec: Hermitian}

For a Hermitian manifold $(M,g,J)$, the Chern curvature $R$ is the curvature induced by the Chern connection $\nabla^C$ which is defined to be the connection such that $\nabla^C g=\nabla^CJ=0 $ with no $(1,1)$ components on the torsion. The first Chern Ricci curvature $\Ric^1$ is defined by
$$\Ric^1=-\ddb \log \det g.$$
Since $\Ric^1$ coincides with the Ricci curvature in the \K case, we can naturally extend the notion of $\mathcal{C}_{\a,\b}$ to general Hermitian metrics:
$$\mathcal{C}_{\a,\b}(X)=\a \Ric^1(X,\ol X)+\b |X|^{-2} R(X,\ol X, X,\ol X).$$
Using this, we show that condition~\eqref{e-negative} is conformally invariant.

\begin{lma}\label{Lemma: Hermitian}
If a Hermitian metric $g$ satisfies
$$\mathcal{C}_{\a,\b}(g)(X)+ \ddb \phi(X,\bar X)\leq \lambda |X|_g^2$$
 for some $\a,\b\in \R$, $\phi\in C^\infty(M)$ and $\lambda\in C^0(M)$ and for all non-zero  $X\in T^{1,0}M$, then $\tilde g=e^{-2F} g$ satisfies
 $$ \mathcal{C}_{\a,\b}(\tilde g)(X)+ \ddb \tilde \phi(X,\bar X)\leq \tilde \lambda |X|_{\tilde g}^2$$
 for $\tilde\phi=\phi+2(n\a+\b) F$, $\tilde\lambda=e^{-2F}\lambda$ and for all non-zero $X\in T^{1,0}(M)$.
\end{lma}
\begin{proof}  Denote the curvature and the first Chern Ricci curvature of the Chern connection simply by $R$ and $\Ric$.
By the conformal formula for Hermitian metric, for example see \cite[Appendix B.1]{LeeTam2020}, the Hermitian metric $\tilde g=e^{2F}g$ satisfies
\begin{equation}
\left\{
\begin{array}{ll}
\tilde R_{k\bar li\bar j}=e^{2F}\left(R_{k\bar li\bar j}-2g_{i\bar j} F_{k\bar l} \right);\\
\tilde R_{i\bar j}=R_{i\bar j}-2n F_{i\bar j}.
\end{array}
\right.
\end{equation}

Hence for $X\in T^{1,0}M$,
\begin{equation}
\begin{split}
 \mathcal{C}_{\a,\b}(\tilde g)(X)
&= \a \widetilde\Ric(X,\bar X)+\b |X|^{-2}_{\tilde g} \tilde R(X,\bar X,X,\bar X)\\
&=\a\left( \Ric_{X\bar X}-2nF_{X\bar X} \right)+|X|^{-2}_g\b \left( R_{X\bar XX\bar X} -2|X|^2_gF_{X\bar X}\right)\\
&= \mathcal{C}_{\a,\b}(g)(X)-2(n\a+\b) F_{X\bar X}.
\end{split}
\end{equation}
The assertion follows.
\end{proof}

By Lemma~\ref{HermitianThm}, it is clear that Theorem~\ref{t-main-1} also holds if the metric $h$ is Hermitian metric which is conformally K\"ahler. This in particular proves a special case of \cite[Conjecture 1.1]{YangZheng2016}.

\begin{thm}\label{HermitianThm}
Suppose $(M,g)$ is a compact  Hermitian manifold which is conformally K\"ahler such that the holomorphic sectional curvature $H_g\leq 0$, then $K_M$ is nef. Moreover if $H_g$ is negative at some point, then $M$ is projective with ample $K_M$.
\end{thm}
\begin{proof}
Since $g$ is conformally K\"ahler,  there is a \K metric $h$ and a smooth function $F$ such that $h=e^{-2F}g$.  The assumption $H_g\le0$ and Lemma~\ref{HermitianThm} imply that
$\mathcal{C}_{0,1}(h)(X)+\ddb \phi(X,\bar X)\leq 0$ for some $\phi\in C^\infty(M)$ and all non-zero $X\in T^{1,0}M$.  Namely, the \K metric $h$ satisfies the assumption in Theorem~\ref{t-main-1}.
 Now, the nefness of $K_M$ follows from {\bf (a)} of Theorem~\ref{t-main-1}.
 Moreover, since $H_g\le0$ and the standard metric of $\mathbb{CP}^1$  is a positive constant, by a result of  Royden \cite[p.558]{Royden1980}, there is no nonconstant holomorphic map from $\mathbb{CP}^1$ into $M$, and so $M$ does not contain  rational curve.



If $H_g$ is negative somewhere, then $\mathcal{C}_{0,1}(h)+\ddb \phi$ is negative at some point by Lemma~\ref{HermitianThm}. The projectivity of $M$ and ampleness of $K_M$ follows from {\bf (c)} of Theorem~\ref{t-main-1} and absence of  rational curve.
\end{proof}

\section{\K manifolds with positive mixed curvature}\label{sec-positive}
In \cite{Ni2019,NiZheng2018-1}, it was shown that compact \K manifolds with $\Ric^+>0$,  $\Ric^\perp>0$ or $\Ric_k>0$ are simply connected and projective. Following the argument in \cite[Theorem 2.7]{Ni2019}, we show that \K manifolds satisfying the positive counterpart of \eqref{e-negative} are also simply connected and projective.
\begin{thm}\label{positive-simplyconnect}
Suppose $(M,g)$ is a compact \K manifold with
\begin{equation}\label{positive-mixed-1}
\a |X|_g^2\cdot \Ric(X,\bar X)+\b R(X,\bar X,X,\bar X)>0
\end{equation}
for some 
{$\a, \b$ with $\a>0,  \a+\b\ge0$} and all non-zero $X\in T^{1,0}M$, then $h^{p,0}=0$ for all $1\leq p\leq n$. In particular, $M$ is simply connected and projective.
\end{thm}
 As a corollary, we recover the above mentioned results by Ni and Zheng \cite{Ni2019,NiZheng2018-1}:
\begin{cor}\label{c-simple-connected} Let $(M^n,g)$ be a compact \K manifold with (i) $\Ric_k>0$ for some $1\le k\le n$; or (ii) $\Ric^+>0$; or (iii) $\Ric^\perp>0$, then $h^{p,0}=0$ for all $1\leq p\leq n$. In particular, $M$ is simply connected and projective.
Here $\Ric^\perp $ is defined as
$$
\Ric^\perp(X,\bar X)=: \Ric(X,\bar X)-\frac{H(X)}{|X|^2}
$$
for all non-zero $X\in T^{1,0}M$.
\end{cor}
\begin{proof} We only need to check the condition \eqref{positive-mixed-1}. If $\Ric_k>0$ for some $1<k<n$, then the condition is satisfied with  $\a=(k-1), \b=(n-k)$ by Lemma \ref{Ric-k-interpolation}. If $k=1$, the condition is satisfied  for small enough $\a>0$ and $\b=1$. If $k=n$, then the condition is satisfied for $\a=1$ and small enough $\b>0$. If $\Ric^+>0$, then the condition is satisfied with $\a=\b=1$. If $\Ric^\perp>0$, then the condition is satisfied with $\a=1, \b=-1$.
\end{proof}

To prove Theorem \ref{positive-mixed-1},  we first show that the Hodge numbers vanish. This will follow from a slight modification of argument in \cite[Section 6]{Ni2019}.
\begin{prop}\label{p-hodge} Suppose $(M,g)$ is a compact \K manifold satisfying \eqref{positive-mixed-1} for some $\a> 0$ and $\b\in\mathbb{R}$ such that $\a+\frac{2\b}{p+1}\geq 0$ for an integer $1\leq p\leq n$, then $h^{p,0}=0$. In particular if $\a+\b\geq 0$, then $h^{p,0}=0$ for all $1\leq p\leq n$.
\end{prop}
\begin{proof}
The first part of proof follows similarly as in that of \cite[Theorem 2.2]{Ni2019}. Assuming the existence of a nonzero holomorphic $(p,0)$-form $\phi$, we may conclude that at the point $x_0$ where the maximum of the comass $\|\phi\|_0$ is attained,
\begin{equation}\label{positive-max-comass}
0\geq \sum_{i=1}^pR_{v\bar vi\bar i}
\end{equation}
for any $v\in T^{1,0}M$, for some choice of unitary frame $\{\frac{\partial}{\partial z^l}\}_{l=1}^n$. Denote $\Sigma=\mathrm{span}\{\frac{\partial}{\partial z^l}: l=1,...,p\}$. 
{On the other hand, assumption \eqref{positive-mixed-1} implies
\begin{equation*}
\begin{split}
0< {} & \fint_{Z\in \Sigma, |Z|=1} \a \Ric(Z,\bar Z)+\b H(Z) \, d\theta(Z)\\
= {} & \frac{\a}{p} \sum_{i=1}^p R_{i\bar i}+\frac{2\b}{p(p+1)} S_p(x_0,\Sigma)
=\frac{\a}{p}\sum_{j=p+1}^n \sum_{i=1}^pR_{i\bar ij\bar j} +\frac1p\left(\a+\frac{2\b}{p+1} \right)\sum_{i,j=1}^p R_{i\bar ij\bar j}.
\end{split}
\end{equation*}
By \eqref{positive-max-comass}, if $\a+\frac{2\b}{p+1}\geq 0$ and $\a\geq 0$, then the right hand side is non-positive which is impossible. This completes the proof.}
\end{proof}

The next ingredient is the compactness of positively curved \K manifold. This was done by the second variational argument in the proof of Bonnet-Meyer theorem.
\begin{prop}\label{l-sc}
Let $(M^n,g)$ be a complete K\"ahler manifold satisfying \eqref{positive-mixed-1} for some $\a, \b, \lambda$ with $\a, \lambda>0$, $\a+\b\ge0$. Then $(M,g)$ is a compact manifold with $\mathrm{diam}(M,g) \leq \pi\sqrt{\frac{\alpha(2n-1)+\beta}{\lambda}}$.
\end{prop}
\begin{proof}
For any $p,q\in M$, let $\gamma:[0,\ell]\rightarrow M$ be a minimizing geodesic connecting $p$ and $q$. It suffices to show $\ell \leq \pi\sqrt{\frac{\alpha(2n-1)+\beta}{\lambda}}$. Let $\{e_{i}\}_{i=1}^{2n}$ be an orthonormal parallel vector fields along $\gamma$ with $e_{2n-1} = J\gamma'$ and $e_{2n} = \gamma'$, where $J$ is the complex structure of $(M,g)$. For $1\leq i \leq 2n-1$, define
\[
V_{i}(t) = \sin\left(\frac{\pi t}{\ell}\right)e_{i}(t), \quad
\phi_{i}(t,s) = \exp_{\gamma(t)}(sV_{i}(t)), \quad
L_{i}(s) = \mathrm{length}(\phi_{i}(\cdot,s)).
\]
Since $\phi_{i}(t,0)=\gamma(t)$ and $\gamma$ is minimizing, then $L_{i}$ has a minimum point at $0$. Using second variation formula of arc length, it is clear that
\begin{equation}\label{diameter eqn 1}
\begin{split}
0 \leq {} & \left.\frac{d^{2}}{ds^{2}}\right|_{s=0}L_{i}(s)
= \int_{0}^{\ell}\left(|\nabla V_{i}|_{g}^{2}-R(V_{i},\gamma',\gamma',V_{i})\right)dt \\
= {} & \int_{0}^{\ell}\left(\left(\frac{\pi}{\ell}\right)^{2}\cos^{2}\left(\frac{\pi t}{\ell}\right)
-\sin^{2}\left(\frac{\pi t}{\ell}\right)R(e_{i},\gamma',\gamma',e_{i})\right)dt.
\end{split}
\end{equation}
Applying \eqref{positive-mixed-1} to $X=\frac{1}{\sqrt{2}}(\gamma'-\sqrt{-1}J\gamma')$, we see that
\begin{equation}\label{diameter eqn 2}
\alpha\Ric(\gamma',\gamma')+\beta R(J\gamma',\gamma',\gamma',J\gamma') \geq \lambda.
\end{equation}
Since $\alpha>0$ and $\a+\beta\ge0$, \eqref{diameter eqn 1} shows
\begin{equation*}
\begin{split}
& 0 \leq \left.\frac{d^{2}}{ds^{2}}\right|_{s=0}\left(\alpha\sum_{i=1}^{2n-2}L_{i}(s)+(\a+\beta) L_{2n-1}(s)\right) \\
= {} & \int_{0}^{\ell}\lf(\alpha(2n-2)+(\a+\beta)\ri)\left(\frac{\pi}{\ell}\right)^{2}\cos^{2}\left(\frac{\pi t}{\ell}\right)dt \\
& -\int_{0}^{\ell}\sin^{2}\left(\frac{\pi t}{\ell}\right)\left(\alpha\sum_{i=1}^{2n-2}R(e_{i},\gamma',\gamma',e_{i})+(\a+\beta) R(e_{2n-1},\gamma',\gamma',e_{2n-1})\right)dt \\
= {} & \frac{\ell}{2}\lf(\alpha(2n-1)+ \beta \ri)\left(\frac{\pi}{\ell}\right)^{2}
-\int_{0}^{\ell}\sin^{2}\left(\frac{\pi t}{\ell}\right)\left(\a\Ric(\gamma',\gamma')+\b R(J\gamma',\gamma',\gamma',J\gamma') \right)dt \\
\leq {} & \frac{\ell}{2}\left((\alpha(2n-1)+\beta)\left(\frac{\pi}{\ell}\right)^{2}-\lambda\right),
\end{split}
\end{equation*}
where we used \eqref{diameter eqn 2} in the last inequality. This completes the proof.
\end{proof}

{\begin{proof}[Proof of Theorem \ref{positive-simplyconnect}] By Proposition \ref{p-hodge}, $h^{p,0}=0$ for $1\le p\le n$. Hence $M$ is projective by a result of Kodaira. By Proposition \ref{l-sc}, the universal cover $\wt M$ is also compact with zero Hodge numbers $\wt h^{p,0}$ for $1\le p\le n$ and is projective. Then one can conclude that $M$ is simply connected by comparing the Euler characteristic numbers of $M$ and $\wt M$ using \cite[Lemma 1]{Kobayashi1961}.
\end{proof}
}


\begin{thebibliography}{10}



\bibitem{Aubin1976}Aubin, T., {\sl \'Equations du type Monge-Amp\`ere sur les vari{\'e}t\'e s k\"ahl\'eriennes compactes} (French), Bull. Sci. Math. (2) 102 (1978), no. 1, 63--95


\bibitem{Campana1992} Campana, F., {\sl Connexit\'e rationnelle des vari\'et\'es de Fano}, Ann. Sci. \'Ecole Norm. Sup. (4),
25(5): 539--545, 1992.








\bibitem{Chu2016} Chu, J., {\sl $C^{2,\a}$ regularities and estimates for nonlinear elliptic and parabolic equations in geometry}, Calc. Var. Partial Differential Equations \textbf{55} (2016), no. 1, Art. 8, 20 pp.


\bibitem{DemaillyPaun2004} Demailly, J.-P.; Paun, M., {\sl  Numerical characterization of the K\"ahler cone of a compact K\"ahler manifold}, Ann. of Math. (2) 159 (2004), no. 3, 1247--1274. MR 2113021


\bibitem{DiverioTrapani2016} Diverio, S.; Trapani, S., {\sl Quasi-negative holomorphic sectional curvature and positivity of the canonical bundle}, J. Differential Geom., 111(2):303--314, 2019.


\bibitem{Evans1982}Evans, L.-C., {\sl Classical solutions of fully nonlinear, convex, second order elliptic equations}, Comm. Pure Appl. Math, 35 (1982), 333--363.




\bibitem{GuedjZeriahi2017} Guedj, V.; Zeriahi, A., {\sl Regularizing properties of the twisted K\"ahler-Ricci flow}, J. Reine Angew. Math. 2017.729 (2017): 275--304.

\bibitem{HeierWong2015} Heier, G.; Wong, B., {\sl On projective K\"ahler manifolds of partially positive curvature and rational connectedness}, Doc. Math. 25 (2020), 219--238.


\bibitem{HeierLuWong2010} Heier, G.; Lu, S.; Wong, B., {\sl On the canonical line bundle and negative holomorphic sectional curvature}, Math. Res. Lett., 17 (2010), no.6, 1101--1110.


\bibitem{HeierLuWong2016} Heier, G.; Lu, S.; Wong, B., {\sl \K manifolds of semi-negative holomorphic sectional curvature}, J. Differential Geom., \textbf{104} (3):419--441, 2016.

\bibitem{HeierLuWongZheng2017} Heier, G.; Lu, S.; Wong, B.; Zheng, F., {\sl Reduction of manifolds with semi-negative holomorphic sectional curvature}, Math. Ann. (2018) 372: 951. https://doi.org/10.1007/s00208-017-1638-8.




\bibitem{HuangLeeTamTong2018}Huang, S.-C; Lee, M.-C.; Tam L.-F.; Tong, F., {\sl Longtime existence of K\"ahler- Ricci flow and holomorphic sectional curvature,} arXiv:1805.12328, appear in Comm. Anal. Geom.

\bibitem{Kobayashi1961} Kobayashi, S., {\sl On compact K\"ahler manifolds with positive Ricci tensor}, Ann. of Math. 74 (1961), 570–574.




\bibitem{KMM1992} Koll\'ar, J.;  Miyaoka, Y.; Mori, S., {\sl Rational connectedness and boundedness of Fano manifolds}, J. Differential Geom., 36(3):765--779, 1992.


\bibitem{Krylov1982} Krylov, N.-V., {\sl Boundedly nonhomogeneous elliptic and parabolic equations}, Izvestia
Akad. Nauk. SSSR 46 (1982), 487-523. English translation in Math. USSR Izv. 20 (1983), no. 3, 459--492.




\bibitem{LeeTam2020}Lee, M.-C.; Tam, L.-F., {\sl Chern-Ricci flows on noncompact complex manifolds}. J. Differential Geom. 115 (2020), no. 3, 529--564.

\bibitem{LeeStreets2019} Lee, M.-C.; Streets, J., {\sl Complex manifolds with negative curvature operator}, Int. Math. Res. Not. IMRN , rnz331, https://doi.org/10.1093/imrn/rnz331


\bibitem{NiElliptic}Li, C., Ni, L.; Zhu. X., {\sl An application of a $C^2$-estimate for a complex Monge–Amp\'ere equation.} International Journal of Mathematics (2021): 2140007.




\bibitem{MM07}   Ma, X.;   Marinescu, G., {\sl Holomorphic Morse inequalities
and Bergman kernels}, volume {\bf 254} of Progress in Mathematics,
Birkh\"auser Verlag, Basel, 2007.

\bibitem{Matsumura2018-1} Matsumura, S., {\sl On morphisms of compact \K manifolds with semi-positive holomorphic sectional curvature}, preprint, arXiv:1809.08859.

\bibitem{Matsumura2018-2}Matsumura, S., {\sl On projective manifolds with semi-positive holomorphic sectional curvature}, preprint, arXiv:1811.04182.

\bibitem{Matsumura2018-3}Matsumura, S., {\sl On the image of MRC fibrations of projective manifolds with semipositive holomorphic sectional curvature}, Pure Appl. Math. Q. 16 (2020), no. 5, 1419--1439.



\bibitem{Ni2019} Ni, L., {\sl The fundamental group, rational connectedness and the positivity of K\"ahler manifolds}, J. Reine Angew. Math. 774 (2021), 267--299.


\bibitem{Ni2018} Ni, L., {\sl Liouville theorems and a Schwarz Lemma for holomorphic mappings between \K manifolds}, Comm. Pure Appl. Math. 74 (2021), no. 5, 1100--1126.


\bibitem{NiZheng2018-1} Ni, L.; Zheng, F., {\sl Comparison and vanishing theorems for \K manifolds}, Calc. Var. Partial Differential Equations, 57(6):Paper No. 151, 31, 2018.

\bibitem{NiZheng2018} Ni, L.; Zheng, F., {\sl Positivity and Kodaira embedding theorem}, preprint, arXiv:1804.09696.




\bibitem{Nomura} Nomura, R., {\sl K\"ahler manifolds with negative holomorphic sectional curvature, K\"ahler-Ricci flow approach}, Int. Math. Res. Not., https://doi.org/10.1093/imrn/rnx075


\bibitem{Royden1980} Royden, H.-L., {\sl The Ahlfors-Schwarz lemma in several complex variables}, Comm. Math. Helv. 55 (1980), no. 4, 547--558.



\bibitem{ShermanWeinkove2012} Sherman, M.; Weinkove, B., {\sl Interior derivative estimates for the K\"ahler-Ricci flow}, Pacific J. Math. 257 (2012), no. 2, 491--501



\bibitem{SongTian2007} Song, J.; Tian, G., {\sl The K\"ahler-Ricci flow on surfaces of positive Kodaira dimension}, Invent. Math. 170 (2007), no. 3, 609--653.

\bibitem{KRFnote}Song, J.; Weinkove, B. , {\sl An introduction to the K\"ahler-Ricci flow. An introduction to the K\"ahler-Ricci flow,} 89--188, Lecture Notes in Math., 2086, Springer, Cham, 2013.



\bibitem{TianZhang2006} Tian, G.; Zhang, Z., {\sl On the K\"ahler-Ricci flow on projective manifolds of general type}, Chinese Ann. Math. Ser. B 27 (2006), no. 2, 179--192.


\bibitem{TosattiNote}Tosatti, V. , {\sl KAWA lecture notes on the K\"ahler-Ricci flow}. Ann. Fac. Sci. Toulouse Math. (6) 27 (2018), no. 2, 285--376.

\bibitem{TosattiWangWeinkoveYang} Tosatti, V.; Wang, Y.; Weinkove, B.; Yang, X.,
{\sl $C^{2,\a}$  estimates for nonlinear elliptic equations in complex and almost complex geometry},  Calc. Var. Partial Differential Equations \textbf{54} (2015), no. 1, 431--453.

\bibitem{TosattiYang2017} Tosatti, V.; Yang, X.-K., {\sl An extension of a theorem of Wu-Yau}, J. Differential Geom., 107(3):573--579, 2017.


\bibitem{Tsuji1988} Tsuji, H., {\sl Existence and degeneration of K\"ahler-Einstein metrics on minimal algebraic varieties of general type}, Math. Ann. 281 (1988), no. 1, 123--133.



\bibitem{WuYau2016} Wu, D.; Yau, S.-T., {\sl Negative Holomorphic curvature and positive canonical bundle}, Invent. Math. 204 (2016), no. 2, 595--604.


\bibitem{WuYau2016-2} Wu, D.; Yau, S.-T., {\sl A remark on our paper “Negative Holomorphic curvature and positive canonical bundle”},  Comm. Anal. Geom., \textbf{24} (2016) no. 4, 901--912.


\bibitem{Yang2018} Yang, X., {\sl RC-positivity, rational connectedness, and Yau's conjecture}, Cambridge J. Math. 6(2018), 183--212.

\bibitem{Yang2018-1} Yang, X., {\sl RC-positive metrics on rationally connected manifolds}, Forum Math. Sigma 8 (2020), Paper No. e53, 19 pp.




\bibitem{Yang2020} Yang, X., {\sl Compact \K manifolds with quasi-positive second Chern-Ricci curvature}, preprint, arXiv:2006.13884.



\bibitem{YangZheng2016} Yang, X.; Zheng, F., {\sl On real bisectional curvature for Hermitian manifolds}, Trans. Amer. Math. Soc. 371, 4 (2019), 2703--2718.



\bibitem{Yau1978-2} Yau, S.-T., {\sl A general Schwarz lemma for K\"ahler manifolds}, Amer. J. Math. 100 (1978), no. 1, 197--203, MR: 0486659.



\bibitem{Yau1978} Yau, S.-T., {\sl On the Ricci curvature of a compact \K manifold and the complex Monge-Amp\`ere equation}, I, Comm. Pure Appl. Math. 31 (1978), no.3, 339--411.


\bibitem{Yau1982} Yau, S.-T., {\sl Problem section. In Seminar on Differential Geometry}, volume 102 of Ann. of Math.
Stud., pages 669--706. Princeton Univ. Press, Princeton, N.J., 1982.


\end{thebibliography}
\end{document}